\documentclass[10pt]{amsart}

\usepackage{verbatim}

\usepackage{amscd,amsmath,amssymb, hyperref}
\usepackage[mathscr]{eucal}
\usepackage{graphicx}

\theoremstyle{plain}

\numberwithin{equation}{section}

\newtheorem{theorem}{Theorem}[section]

\newtheorem*{maintheorem*}{Main Theorem}
\newtheorem{lemma}[theorem]{Lemma}

\newtheorem{corollary}[theorem]{Corollary}

\theoremstyle{definition}
\newtheorem{definition}[theorem]{Definition}
\newtheorem{example}[theorem]{Example}
\newtheorem{remark}[theorem]{Remark}
\newtheorem*{claim*}{Claim}
\newtheorem*{example*}{Example}
\newtheorem*{remark*}{Remark}
\newcommand{\R}{\mathbf{R}}
\newcommand{\Rn}{\R^n}

\DeclareMathOperator{\vol}{vol}

\newcommand{\dom}{\mathop{\rm Dom}}

\newcommand{\0}{{0}}

\newcommand\cexp{\hbox{$c$-Exp}}
\newcommand\cBexp{\hbox{$c_B$-Exp}}
\newcommand\cMexp{\hbox{$c_M$-Exp}}

\newcommand\dist{\hbox{\rm dist}}
\newcommand\cutlocus{\hbox{\rm cut-locus}}

\newcommand\cross{\hbox{\rm cross}}
\newcommand\qec{\hfill $\triangle$ \medskip}

\begin{document}

\title[curvature of optimal transport]
{Towards the smoothness of optimal maps on Riemannian submersions and Riemannian products
 (of round spheres in particular)}

\author{Young-Heon Kim
and Robert J. McCann
}

\address{Department of Mathematics, University of Toronto\\
  Toronto, Ontario Canada M5S 2E4}

\email{yhkim@math.toronto.edu,
mccann@math.toronto.edu
}
\date{\today}

\thanks{This research was supported in part by Natural Sciences and Engineering
Research Council of Canada Grant 217006-03 and United States National Science Foundation
Grant DMS-0354729.\\ \indent
\copyright 2008 by the authors.}


\begin{abstract}
The key condition {\bf A3w} of Ma, Trudinger and Wang for regularity of optimal
transportation maps is implied by the nonnegativity of a pseudo-Riemannian curvature ---
which we call \emph{cross-curvature} --- induced by the transportation cost.
For the Riemannian distance squared cost,
it is shown that (1) cross-curvature nonnegativity
is preserved for products of two manifolds; (2) both {\bf A3w} and cross-curvature
nonnegativity are inherited by Riemannian submersions; and (3) the $n$-dimensional round
sphere satisfies cross-curvature nonnegativity.
From these results, a large new class of  Riemannian manifolds satisfying
cross-curvature nonnegativity (thus {\bf A3w}) is obtained, including many whose sectional
curvature is far from constant. All known obstructions to the regularity of optimal maps
are absent from these manifolds, making them a class for which it is natural to conjecture
that regularity holds.  This conjecture is confirmed for
complex projective space $\mathbf{CP}^n$.
\end{abstract}
\maketitle

\section{Introduction}
This paper addresses questions in optimal transportation theory and Riemannian geometry.
For a general introduction to these subjects we refer to the books by
Villani \cite{V} \cite{V2} for optimal transport theory and the book by
Cheeger and Ebin \cite{CE} for Riemannian geometry.

\subsection{Background: optimal transport and pseudo-Riemannian geometry}
In optimal transportation theory one is interested in phenomena which occur when moving mass
distributions so as to minimize the transportation cost.
Mathematically, there are source and target domains, $M$, $\bar M$, two differential
manifolds equipped with a lower semi-continuous cost function
$c: M \times \bar M \to \R \cup \{\infty\}$. Given two  positive Borel probability measures
$\rho$, $\bar \rho$ on $M$, $\bar M$, respectively, one wants to understand the
optimal map $F: M \to \bar M$, which minimizes the average cost
\begin{align*}
\int_M c(x, F(x)) d\rho(x)
\end{align*}
among all such maps pushing-forward $\rho$ to $\bar \rho$,
denoted $F_\# \rho = \bar \rho$ and meaning that
\begin{align*}
\rho (F^{-1} (E)) = \bar\rho (E) \qquad \forall \ \hbox{Borel} \ E \subset \bar M.
\end{align*}

Particular attention has been devoted to the case $M=\bar M$, a complete
Riemannian manifold with
the cost $c=\frac{1}{2}\dist^2$, where $\dist$ denotes the Riemannian distance function.
Existence and uniqueness of optimal maps in this case is well known due to the work of
Brenier \cite{B} and McCann \cite{M} (and Figalli \cite{F} for noncompact manifolds),
under the condition that $\rho$  doesn't charge
lower dimensional submanifolds (see also \cite{ChiapporiMcCannNesheim07p}).
Under suitable hypotheses,
regularity ($C^0/C^{\alpha}/C^\infty$) of such optimal maps is known for Euclidean
space $M=\Rn$ by results of Delano\"e \cite{D}, Caffarelli \cite{C} \cite{C2},
and Urbas \cite{U}, and for flat \cite{Co} and near flat \cite{D2}
manifolds by Cordero-Erausquin and Delano\"e. Beyond the flat case,
Loeper \cite{L2} deduced regularity on the round sphere, by combining
his own breakthroughs with pioneering results of
Ma, Trudinger and Wang \cite{MTW}\cite{TW}\cite{TW2}
concerning regularity of optimal maps for general cost functions.
Loeper's result
is simplified in the work of present authors 
\cite{KM2}
where we give an elementary
and direct proof of a crucial maximum principle deduced by Loeper from Trudinger and
Wang's theory. We referred to Loeper's maximum principle with the acronym
{\bf DASM} (see Theorem~\ref{T:cross and DASM}), and extended it to the manifold case.
In the course of deriving new results,
our method was employed and further developed by Figalli and Villani \cite{FV}, Figalli and Rifford \cite{FR},
Loeper and Villani \cite{LV}, and Villani \cite{V2} \cite{V3}.

Our contributions in \cite{KM2} are based on a pseudo-Riemannian geometric structure $h$ on $M \times \bar M$ induced by the transportation cost.
Namely when $\dim M = \dim \bar M$ this pseudo-metric $h$ is defined on
$N \subset M \times \bar M$ with $c \in C^4(N)$ as the following
symmetric 
bilinear form on $TM \oplus T\bar M$: \begin{equation}\label{metric h}
h := \left(
\begin{array}{cc}
0 & -\frac{1}{2}\bar D D c \\
-\frac{1}{2}D \bar D c & 0
\end{array}
\right)
\end{equation}
where
$D$ and $\bar D$ denote the differentials along $M$ and $\bar M$, respectively.
For non-degeneracy of $h$ we assume that $D\bar D c$ and its adjoint $\bar D D c$ are
non-degenerate (Ma, Trudinger and Wang's condition {\bf A2}: see Section~\ref{S:pre}).
This is automatically true for the Riemannian distance squared cost, away from the cut-locus,
since $\bar D D c$ is a matrix of independent Jacobi fields.
This pseudo-metric $h$ geometrizes the regularity theory of optimal maps, by recasting the
key cost hypothesis {\bf A3w} of Ma, Trudinger and Wang \cite{MTW}\cite{TW} for regularity
as the non-negativity of certain pseudo-Riemannian sectional curvatures of $h$. To explain this condition more precisely, we need
\begin{definition}[\bf cross-curvature]\label{D:cross-curv}
Let $(x, \bar x) \in N \subset M \times \bar M$. For each $p\oplus \bar p \in T_{(x,\bar x)} N = T_x M \oplus T_{\bar x} \bar M$, \emph{the cross-curvature} of $p$ and $\bar p$ is defined as
\begin{align}\label{cross-curv}
\cross_{(x,\bar x)}^{(N,h)} (p, \bar p) = R_h ((p\oplus 0 )\wedge (0 \oplus \bar p), (p\oplus 0) \wedge (0 \oplus \bar p))
\end{align}
where $R_h$ is the Riemann curvature operator of the pseudo-metric $h$. We drop the superscript $(N,h)$ and the subscript $(x, \bar x)$ when no ambiguity can occur.
\end{definition}
\noindent Ma, Trudinger and Wang's {\bf A3w} condition and its strict version {\bf A3s}
then assert
\begin{align}\label{A3w}
\hbox{{\bf A3w:}} &  \ \ \ \ \ \hbox{$\cross (p, \bar p) \ge 0$ for all $p \oplus \bar p$ with $h(p\oplus \bar p, p \oplus \bar p) = 0$;}\\ \label{A3s}
\hbox{{\bf A3s:}} & \ \ \ \ \ \hbox{in addition, $\cross(p, \bar p) = 0$ in (\ref{A3w})
implies  $p =0$ or $\bar p =0$}.
\end{align}
{\bf A3w} and {\bf A3s} are also called \emph{weak regularity} and \emph{strict regularity}, respectively.
If $(N, h)$ satisfies the inequality $\cross(p,\bar p) \ge 0$ for all $p \oplus \bar p$ (whether or not
$h(p\oplus\bar p, p\oplus\bar p )$ vanishes), then $(N,h)$ is said to be \emph{non-negatively cross-curved}.
Loeper \cite{L} showed that {\bf A3w} is necessary and {\bf A3s} sufficient
for continuity of optimal maps between suitable measures: without {\bf A3w}
there are discontinuous optimal maps between smooth measures $\rho$, $\bar \rho$
on nice domains.
Trudinger and Wang \cite{TW} had already shown the sufficiency of {\bf A3w}
for
continuity (indeed smooth differentiability) of optimal mappings, under much stronger smoothness and convexity restrictions
on $\rho$ and $\bar \rho$.
These restrictions on $\rho$, $\bar \rho$ are relaxed in two dimensions by Figalli and Loeper \cite{FL}, where they showed a continuity result of optimal maps under {\bf A3w}. It still remains an open question to show such continuity result in higher dimensions.
However, in a separate work with Figalli \cite{FKM}, we show non-negative
cross-curvature allows these 
smoothness and convexity restrictions to be relaxed without sacrificing
H\"older continuity of optimal maps and without assuming {\bf A3s},
thus obtaining a continuity theory for smooth cost functions
parallelling and extending Caffarelli's $c(x,y) = \frac{1}{2} |x-y|^2$ result \cite{C}.


For the Riemannian distance squared cost $c(x, \bar x) = \mathop{\rm dist}^2 (x, \bar x)/2$
on a Riemannain manifold $M = \bar M$, an isometric copy of
$M$ is embedded totally geodesically as the diagonal of
$M \times M$ with respect to the pseudo-Riemannian metric $h$.
Along this diagonal,
\begin{align*}
\cross_{(x,x)} (p, \bar p) = \frac{4}{3}
R_M (p \wedge \bar p, p \wedge \bar p)
\end{align*}
where $R_M$ denotes the curvature operator of $M$
(see \cite{KM2} for details).  This provides some geometric
intuition motivating Loeper's result \cite{L} that
{\bf A3w} implies nonnegative sectional curvature of the Riemannian metric. (However, the nonnegative/positive curvature does not imply {\bf A3w}, as shown by a counterexample in \cite{K}.)
Loeper also verified {\bf A3s} for the standard round sphere \cite{L2} and
used it to obtain $C^{1/\max\{5,4n-1\}}$ and $C^\infty$ regularity results for optimal maps
in this spherical setting.

\subsection{Main results}
Throughout this paper, if not specified, each Riemannian manifold $M$ is assumed to be complete and to be
equipped with the cost function  $c= \frac{1}{2}\dist^2$ and this cost induces the
pseudo metric $h$ on $N= M \times M \setminus \cutlocus$. A Riemannian
manifold $M$ is said to be {\bf A3w} / {\bf A3s} / non-negatively cross-curved if $(N, h)$
satisfies the corresponding cross-curvature condition.

Our main results provide methods of generating new examples of Riemannian manifolds
which satisfy  non-negative cross-curvature and thus {\bf A3w}.
As announced in \cite{KM2}, we show:

\begin{theorem}{\bf (Products and submersions, particularly of round spheres)}\label{T:main}

\noindent 1) $S^n$ with its standard round metric is non-negatively cross-curved.

\noindent 2) Let $\pi: \tilde M \to M$ be a Riemannian submersion (see Definition~\ref{D:submersion}).
If $\tilde M$ is {\bf A3w} / {\bf A3s} / non-negatively cross-curved then so is $M$.

\noindent 3) For product $M_+ \times M_-$ of Riemannian manifolds, if each factor
$M_\pm$ is  non-negatively cross-curved, then the resulting manifold $M_+ \times M_-$
is non-negatively cross-curved, thus {\bf A3w} holds (but never {\bf A3s}).

\noindent 4) Moreover, if either of the factors above fail to be non-negatively cross-curved then the
product $M_+ \times M_-$ fails to be {\bf A3w}.
\end{theorem}

\begin{proof}
The proof of assertion 1) is a calculation given in Section~\ref{S:sphere cross-curv}.
Assertion 2) is shown in Section~\ref{S:submersion}. Assertion 3) and 4) are easy facts
which are explained in Section~\ref{S:product} in detail.
\end{proof}

\begin{remark}
Following our announcement of Theorem~\ref{T:main} \cite{KM2},
Figalli and Rifford \cite{FR} gave an alternate proof of result 1)
in a different and slightly stronger form.
\end{remark}

As a byproduct of our method, we obtain an O'Neill type inequality for cross-curvature in
Riemannian submersions (see Theorem~\ref{T:cross and submersion}). This verifies that
Riemannian submersion quotients of the round sphere all satisfy {\bf A3s}, from which
the regularity of optimal maps follows on notable examples such as
on complex projective spaces $\mathbf{CP}^n$ with the Fubini-Study metric:
see Section~\ref{S:regularity and submersion}.
This is a new result which is not covered
by other discrete quotient cases of Delano\"e and Ge \cite{DG} or Figalli and Rifford \cite{FR}.

As another important consequence, the
Riemannian product of round spheres and Euclidean space
${\mathbf S}^{n_1} \times \cdots \times {\mathbf S}^{n_k}\times \R^l$ and its
\emph{Riemannian submersion quotients}, all satisfy cross-curvature non-negativity and
thus {\bf A3w}.
Since this rules out the known counterexamples to regularity
\cite{L}, optimism combines with a lack of imagination to lead to the conjecture
that regularity of optimal mappings also holds in such settings.
Together with the perturbations 
\cite{DG}\cite{LV}\cite{FR}
of the round sphere and its discrete quotients discussed by 
Delano\"e and Ge, Loeper and Villani,
and Figalli and Rifford,  for which the continuity \cite{FR} or regularity \cite{LV}\cite{DG} of optimal maps is already shown,
these presently form the only examples of non-flat Riemannian manifolds
which are known to be {\bf A3w}. (In flat manifold case, {\bf A3w} is trivial, and in fact the cross-curvature vanishes everywhere. Regularity of optimal maps in this flat case is known by Cordero-Erausquin \cite{Co} applying Caffarelli's result \cite{C}\cite{C2}.)

As far as we know, it remains an
open challenge to show regularity of optimal maps on the new tensor product type examples
${\mathbf S}^{n_1} \times \cdots \times {\mathbf S}^{n_k}\times \R^l$
(when the supports of the source and target measure are the whole domain).
However, Loeper's maximum principle (and a stronger convexity statement)
is easily verified on these examples
using the results and methods of \cite{KM2}: see
Corollary~\ref{C:cross and global time-convex} and Remark~\ref{R:products and DASM}.
We hope these key ingredients will make it possible to address the regularity issue
in a subsequent work.



This paper is organized as follows: Section~\ref{S:pre} is devoted to preliminary notions
and facts. Some important geometric implications of cross-curvature non-negativity
(Theorem~\ref{T:cross and local time-convex} and
Corollary~\ref{C:cross and global time-convex}) are shown also.
In Section~\ref{S:product}, the tensor product construction of costs is explained.
Section~\ref{S:submersion} discusses the relation between Riemannian submersion and
cross-curvature. Notably, we derive an O'Neil type inequality for cross-curvature
(see Theorem~\ref{T:cross and submersion}). Section~\ref{S:regularity and submersion} shows regularity of optimal maps on the complex projective spaces ${\bf CP}^n$.
Section~\ref{S:sphere cross-curv}
shows the cross-curvature non-negativity of the standard round
sphere.

\subsection*{Acknowledgements}
We thank
Philippe Delano\"e, Yuxin Ge, Alessio Figalli,
Gr\'egoire Loeper, Neil Trudinger, Cedric Villani and Xu-Jia Wang for useful discussions and timely exchanges of preprints. We are also grateful to
the participants of Fields Analysis Working Group Seminar 2006--2008 for the
stimulating research environment they helped to create.

\section{Preliminaries}\label{S:pre}
In this section, we recall basic terminology and facts from \cite{KM2}
(also \cite{MTW} \cite{TW} \cite{L}). Theorem~\ref{T:cross and local time-convex} and
Corollary~\ref{C:cross and global time-convex} are not stated there, but are
easy consequences of the method in \cite{KM2}, to which we also refer the reader for proofs
of the other results summarized below. For improvements to these results, see also \cite{V2}.
Let $N \subset M \times \bar M$ be an open set
where the cost function $c \in C^4 (N)$.

\begin{definition}[\bf visible sets, twist condition and non-degeneracy]\label{D:A1,2}We define the visible sets:
\begin{align*}
N(\bar x) &= \{ x \in M \ | \ (x, \bar x) \in N \},\\
\bar N (x) &= \{ \bar x \in \bar M \ | \ (x, \bar x) \in N \}.
\end{align*}
Throughout this section it is assumed as in \cite{MTW} that for all $(x, \bar x) \in N$,
\begin{align*}
\hbox{{\bf A1}} & \ \ \ \ \ \ Dc( x, \cdot ): \bar N (x) \to T^*_x M , \
\bar D c( \cdot, \bar x ): N (\bar x) \to T^*_{\bar x} \bar M \hbox{ are injective;} \\
\hbox{{\bf A2}} & \ \ \ \ \ \ D\bar D c \hbox{ is non-degenerate}.
\end{align*}
\noindent We also call {\bf A1} the \emph{(bi-)twist condition}
and {\bf A2} \emph{non-degeneracy}.
\end{definition}

We recall an important map 
of Ma, Trudinger \& Wang
\cite{MTW}, 
called the {\em cost-exponential} by Loeper \cite{L}, which coincides with the Riemannian exponential map for $c= \frac{1}{2}\dist^2$.

\begin{definition}\label{D:c-exp}{\bf (cost exponential)}
If $c \in C^2(N)$ is twisted {\bf (A1)}, we define the {\em $c$-exponential} on
\begin{align}\label{dom c-exp}\nonumber
\dom(c\text{-Exp}_{x} ) :=& - D c(x, \bar N(x)) \\
=& \{ p^* \in T^*_{x} M \  | \ p^* = - D c(x, \bar x) \hbox{ \ \ for some $\bar x \in \bar N(x)$}  \}
\end{align}
by $\cexp_{x} p^* = \bar x$ if $p^*=- D c(x,\bar x)$.
Non-degeneracy {\bf (A2)} then implies the $c$-exponential is a diffeomorphism from
$\dom(\cexp_{x}) \subset T^*_{x} M$ onto $\bar N(x) \subset \bar M$.
\end{definition}

\begin{remark}[\bf differential of $\cexp$]\label{D:differential cexp}
Linearizing the cost exponential $\bar x = \cexp_x p^*$ around $p^* \in T_x^*M$ we obtain
a map
$\cexp_{(x, \bar x)*}:T_x^* M \longrightarrow T_{\bar x} \bar M$
given explicitly by
\begin{align}\label{differntial cexp}
\hbox{$\cexp_{(x, \bar x (t) )*} (\dot{p}^* (t)) =  \dot{\bar x}(t) $ for $p^*(t)= -Dc(x, \bar x(t))$.}
\end{align}
Equivalently $\cexp_{(x, \bar x)*} = -\bar D D c(x,\bar x)^{-1}$,
meaning the inverse tensor  $\frac{1}{2}h^{-1}$ to the metric \eqref{metric h}
--- which gives the pseudo-Riemannian correspondence between the tangent and cotangent
spaces to $N \subset M \times \bar M$ --- also carries covectors forward through the
cost-exponential.
\end{remark}

\noindent The following lemma characterizes Ma, Trudinger \& Wang's
\emph{$c$-segments} as geodesics of $h$.
\begin{lemma}\label{L:c-segments are geodesics}
{\bf (the $c$-segments of \cite{MTW} are geodesics)}
Use a twisted ({\bf A1}) and non-degenerate ({\bf A2}) cost $c \in C^4(N)$ to define a pseudo-metric
(\ref{metric h}) on the domain $N \subset M \times \bar M$.
Fix ${x} \in M$.
For each line segment $(1-s)p^* + sq^* \in \dom(\cexp_x)$, $s\in [0,1]$, the curve
$$s \in [0,1] \longrightarrow %
\sigma(s) := (x,\cexp_{x} ((1-s)p^* + s q^*))$$
is an affinely parameterized null geodesic in (N,h).
Conversely, every geodesic segment in the totally geodesic submanifold $\{x\} \times \bar N(x)$
can be parameterized locally in this way.
\end{lemma}

To see some relevant geodesic equations in local coordinates,  given  $s_0 \in [0,1]$,
introduce coordinates on $M$ and $\bar M$ around $\sigma(s_0)$ so that nearby, the curve
$\sigma(s)$ can be represented in the form $(x^1,\ldots,x^n,x^{\bar 1}(s),\ldots,x^{\bar n}(s))$.
Differentiating the definition of the cost exponential
\begin{equation}\label{c-segment}
0 = (1-s)p^*_i + sq^*_i + c_i(\sigma(s))
\end{equation}
twice with respect to $s$ yields
\begin{equation}\label{h-geodesic equation}
0 
= c_{i {\bar j}} \ddot x^{\bar k} + c_{i {\bar j} {\bar k} } \dot x^{\bar j} \dot x^{\bar k}
\end{equation}
for each $i=1,\ldots,n$. This equation will be used later.

Regarding curvature of the metric $h$, the following fact is fundamental.
\begin{lemma}\label{L:mixed fourth derivative}
{\bf (Non-tensorial expression for curvature)}
Use a non-degenerate cost $c \in C^4(N)$ to define a pseudo-metric
\eqref{metric h} on the domain $N \subset M \times \bar M$.
Let $(s,t) \in \mathopen[-1,1\mathclose]^2 \longrightarrow (x(s),\bar x(t)) \in N$
be a surface
containing two curves $\sigma(s)=(x(s), \bar x(0))$ and
$\tau(t) = (x(0),\bar x(t))$ through 
$(x(0),\bar x(0))$.  Then $\0 \oplus \dot {\bar x}(0)$ defines a parallel vector-field along
$\sigma(s)$. Moreover, if $s \in [-1,1] \longrightarrow \sigma(s) \in N$ is a geodesic in $(N,h)$
then
\begin{equation}\label{mixed fourth derivative}
- 2 \frac{\partial^4}{\partial s^2 \partial t^2}\bigg|_{s=0=t}
c(x(s),\bar x(t)) = \cross_{(x(0),\bar x(0))} (\dot x (0), \dot{\bar x} (0)).
\end{equation}
\end{lemma}
\noindent Note in this lemma that only one curve $\sigma(s)$ needs to be geodesic in $(N, h)$.
As a consequence of this result, the conditions {\bf A3w/s} can alternately be
characterized by the concavity / $2$-uniform concavity
for each $x \in M$ and $q_0^* \in\dom(c\text{-Exp}_{x} )$
of the function
\begin{align}\label{Alternative A3}
q^* \in T_x^* M \longrightarrow p^i p^j c_{ij} (x, \cexp_x (q_0^* + q^*)).
\end{align}
restricted to $q^*$ in the nullspace of $p \in T_x M$.
Cross-curvature nonnegativity asserts this concavity extends to all $q^*$
(not necessarily in the nullspace of $p$) such that
$q_0^* + q^* \in \dom (c\text{-Exp}_{x} )$.

Before recalling the important geometric implications of the curvature properties of $h$,
let us define:
\begin{definition}\label{D:visible set}{\bf (Illuminated set)}
Given $(x,\bar x) \in N$,  let $V(x,\bar x) \subset M$ denote those points
$y \in N(\bar x)$
for which there exists a geodesic curve from $(x,\bar x)$ to $(y,\bar x)$ in $N(\bar x) \times \{\bar x\}$.
\end{definition}

We now state a version of Loeper's maximum principle.  Although the statements
below are from \cite{KM2},  where a direct elementary proof is given
using pseudo-Riemannian geometry,
the fundamental form of the equivalence (Theorem~\ref{T:cross and DASM}) below
was deduced by Loeper \cite{L} in the simpler setting
$N = M \times \bar M \subset \R^n \times \R^n$
from results of Trudinger and Wang \cite{TW}\cite{TW2}.
We visualize his maximum principle as asserting that the
{\em double-mountain} $\max[f_0,f_1]$ stays {\em above} the {\em sliding mountain}
$f_t(y):=-c(y, \bar x(t)) + c(x, \bar x(t))$, hence refer to it by the acronym DASM.

\begin{theorem}[\bf A3w $\Leftrightarrow$ local DASM]\label{T:cross and DASM} Let $h$ be the pseudo-Riemannian metric on $N \subset M \times \bar M$ induced from the non-degenerate cost $c \in C^4(N)$ as in \eqref{metric h}. The following are equivalent.

\noindent 1. $(N, h)$ satisfies {\bf A3w}.

\noindent 2. {\bf (local DASM)} For any $h$-geodesic
$\sigma:t \in [0,1] \to (x, \bar x(t)) \in N$ and sufficiently small neighbourhood
$U \subset V_\sigma$ of $x$,  where
$V_\sigma:=\cap_{0 \le t \le 1} V(x, \bar x(t))$ is from
Definition~\ref{D:visible set},
the sliding mountain $f_t(y):=-c(y, \bar x(t)) + c(x, \bar x(t))$
satisfies the maximum principle
\begin{align*}
\hbox{
$f_t (y) \le \max [f_{0} , f_{1}] (y)$ \quad for $ 0 \le  t \le  1$.
}
\end{align*}
\end{theorem}

Then one can prove as in \cite{KM2} the following. \begin{theorem}[\bf A3w + $c$-convexity of domains $\Rightarrow$ DASM]\label{T:A3w and DASM}
Let $h$ be the pseudo-Riemannian metric on $N \subset M \times \bar M$ induced by the
non-degenerate cost $c \in C^4(N) \cap C(M \times \bar M)$ as in \eqref{metric h}.
 Suppose $(N, h)$ is {\bf A3w}, and the set $V_\sigma :=\cap_{0 \le t \le 1} V(x, \bar x(t))$
from Definition~\ref{D:visible set} is dense in $M$ for some $h$-geodesic
$\sigma: t \in [0,1] \longrightarrow (x, \bar x(t)) \in N$.
For any $y$ in $M$, 
the sliding mountain $f_t(y):=-c(y, \bar x(t)) + c(x, \bar x(t))$ satisfies the
maximum principle
\begin{align}\label{DASM}
\hbox{{\bf DASM:} \ \ \ \ \ \
$f_t (y) \le \max [f_{0} , f_{1}] (y)$ \quad for $0 \le t \le 1$.
}
\end{align}
\end{theorem}
\noindent
\begin{remark}[\bf Relating density of the illuminated set to $c$-convexity of the domain]
In case $N= M \times \bar M$,  the density hypothesis of the preceding theorem holds
whenever $\bar D c(M,\bar x) \subset T^*_{\bar x}\bar M$ is convex
for each $\bar x \in \bar M$, since then $V_\sigma = M$.
This is the case considered initially by Trudinger and Wang \cite{TW} and Loeper \cite{L}.
For our argument, it is enough that $\bar D c(M,\bar x(t))$ be star-shaped around
$\bar Dc(x,\bar x(t))$ for each $t \in [0,1]$.
Conversely,  $V(x,\bar x)=M$ for each $x \in M$ implies convexity of $\bar Dc(M,\bar x)$.
\end{remark}

Our pseudo-Riemannian method \cite{KM2} makes it equally possible to deduce further
geometric implications of the cross-curvature condition, as the next theorem
and corollary show.

\begin{theorem}[\bf nonnegative cross-curvature $\Leftrightarrow$ local time-convex sliding mountain]\label{T:cross and local time-convex} Let $h$ be the pseudo-Riemannian metric on $N \subset M \times \bar M$ induced from the non-degenerate cost $c \in C^4(N)$ as in \eqref{metric h}. The following are equivalent.

\noindent 1. $(N, h)$ is non-negatively cross-curved.

\noindent
2. For each $h$-geodesic $\sigma:t \in [0,1] \longrightarrow (x, \bar x(t)) \in N$
and sufficiently small neighborhood $U \subset V_\sigma$ of $x \in U$, where
$V_\sigma:=\cap_{0 \le t \le 1} V(x, \bar x(t))$
is from Definition~\ref{D:visible set}, 
the sliding mountain $f_t(y):=-c(y, \bar x(t)) + c(x, \bar x(t))$ is a convex function
of $t \in [0,1]$ for each $y \in U$, i.e.,
$\frac{\partial^2}{\partial t^2} f_t (y) \ge 0$ for $0 \le t \le 1$.
\end{theorem}

\begin{proof}[Proof: (1 implies 2)] Fix an arbitrary $h$-geodesic
$t:[0,1] \to (x, \bar x(t)) \in N $.
Let $U$ be chosen open so that
$x \in U \subset V_\sigma$.
The existence of such a $U$ is elementary, tedious, and independent of hypothesis 1:
it requires checking --- at least for $\delta$
sufficiently small depending on $T \in [0,1]$ ---
that $\cap_{T-\delta \le t \le T+\delta} V(x,\bar x(t))$ contains a neighbourhood of $x$,
as we now do.  Choose coordinates on $M$ near $x$.
Since the geodesic $\sigma$ is compact in the open set $N$, some coordinate ball
satisfies $B_r(x) \times \{\bar x(t)\} \subset N$ for all $t \in [0,1]$.
Since the coordinate charts
$x \in M \longrightarrow \bar D c(x, \bar x(t)) \in T^*_{\bar x(t)} \bar M$
are $C^2$ smooth functions of $(x,t) \in M \times [0,1]$, taking $r=r(T)$ and $\delta(T)>0$
sufficiently small ensure $\bar D c(B_r(x),\bar x(t))$ is convex for each $t$ within
$\delta(T)$ of $T$.
This convexity implies $B_r(x) \subset V(x,\bar x(t))$.
Extracting a finite subcover $\cup_{i=1}^N (T_i - \delta(T_i),T_i+\delta(T_i))$ of $[0,1]$
and taking $U = \cap_{i \le N} B_{r(T_i)}(x)$ will suffice.

Now define
$f_t(\cdot ):= -c (\cdot, \bar x(t))+ c(x, \bar x(t))$.  Fix arbitrary $t_0 \in [0,1]$ and
$y \in U$. The fact $U \subset V_\sigma$ guarantees an $h$-geodesic
$s: [0,1] \to (x (s), \bar x(t_0)) \in N$ with $x (0)=x$ and $x (1)=y$.
Define an auxiliary function $g(s):= \frac{\partial^2}{\partial t ^2}\big{|}_{t=t_0} f_t (x(s))$, which shall be shown to be non-negative for $s \in [0,1]$.
By Lemma~\ref{L:mixed fourth derivative} and the cross-curvature non-negativity,
\begin{align}\label{f nonnegative mixed fourth derivative}
\frac{d^2 g}{d s^2}  \ge 0, \hbox{  for $s \in [0,1]$}.
\end{align}
In particular, $g(s)$ is convex.
It is clear that $g(0)=\frac{\partial^2}{\partial t^2}\big{|}_{t=t_0} f_t (x)=0$.
We also claim $g'(0) =0$: introducing
coordinates $x^1,\ldots,x^n$ around $x=x(0)$ on $M$ and
$x^{\bar 1},\ldots,x^{\bar n}$ around $\bar x(t_0)$ on $\bar M$, we compute
\begin{eqnarray}\label{vanishing due to geodesy}
 \frac{d g}{d s }\bigg|_{s=0}
&=& - (c_{{\bar i} k}(x(0),\bar x(t_0)) \ddot x^{\bar i} + c_{{\bar i} {\bar j} k}(x(0),\bar x(t_0)) \dot x^{\bar i} \dot x^{\bar j}) \dot x^k\\\nonumber
&=& 0
\end{eqnarray}
by the $h$-geodesic equation \eqref{h-geodesic equation} for $t \in [0,1] \to (x, \bar x (t)) \in N$.
From \eqref{f nonnegative mixed fourth derivative}, this shows $g(s) \ge 0$, for $0 \le s \le 1$, in particular at $s=1$,
\begin{align*}
g(1)= \frac{\partial^2}{\partial t^2}f_t (y)\Big{|}_{t=t_0} \ge 0.
\end{align*}
Since we have used $U \subset V_\sigma$ but not the openness or smallness of $U$,
we have actually deduced convexity of $t \in [0,1] \longrightarrow f_t(y)$ for all
$y \in V_\sigma$.  A fortiori, 1 $\Longrightarrow$ 2.
\qec

\noindent
\emph{Proof: (2 implies 1).} Fix $(x, \bar x) \in N$, $p \oplus \bar p \in T_{(x, \bar x)} N$.
It shall be shown that $\hbox{sec}_{(x, \bar x)} ((p\oplus 0) \wedge (0 \oplus \bar p) ) \ge 0$.
Choose an $h$-geodesic $t \in [-1,1] \to (x, \bar x(t)) \in N$, with $\bar x(0)= \bar x$,
$\dot {\bar x}(0)= \bar p$, and a curve $ s \in [-1, 1] \to (y(s), \bar x) \in N$, with
$y(0)=x$, $\dot y (0)= p$.
Let $f(\cdot):= -c (\cdot, \bar x (t)) + c(x, \bar x(t))$.
Suppose
\begin{align}\label{g nonnegative}
g(s):=\frac{\partial^2}{\partial t ^2}\Big{|}_{t=0} f_t (y (s)) \ge 0,
\end{align} for $s \in [-1, 1]$: this holds if  the property 2 is assumed and the curve
$y(s)$ is chosen inside the neighborhood $U$ of $x$ constructed at the outset.
Note $g(0) = \frac{\partial ^2}{\partial t ^2}\big{|}_{t=0} f_t (x) =0$.
Thus, from \eqref{g nonnegative}, $g'( 0) =0$ and
\begin{align*}
0 & \le \frac{d ^2}{d s ^2} \Big{|}_{s=0} g (s) \\
& = \frac{1}{2}\hbox{cross}_{(x, \bar x)} (\dot y (0), \dot {\bar x} (0)).
\end{align*}
The last equality comes from Lemma~\ref{L:mixed fourth derivative}.
This completes the proof 2 $\Longrightarrow$ 1.
\end{proof}


\begin{corollary}[\bf non-negative cross-curvature $+$ $c$-convexity of domains $\Longrightarrow$ time-convex sliding mountain]\label{C:cross and global time-convex}
Suppose the pseudo-Riemannian metric $h$ induced by the
non-degenerate cost $c \in C^4(N) \cap C(M \times \bar M)$ on $N \subset M \times \bar M$
in \eqref{metric h} is non-negatively cross-curved.
If for some $h$-geodesic
$\sigma:t \in [0,1] \longrightarrow (x, \bar x(t)) \in N$ the set
$V_\sigma :=\cap_{0 \le t \le 1} V(x, \bar x(t))$ from Definition~\ref{D:visible set}
is dense in $M$,  then
for each $y \in M$, the sliding mountain $f_t(y):=-c(y, \bar x(t)) + c(x, \bar x(t))$
satisfies 
 \begin{align}\label{global time-convex}
 f_t (y) \le (1-t) f_0(y) + t f_1(y) \quad {\rm for}\ 0 \le t \le 1.
 \end{align}
\end{corollary}

\begin{proof}
For $y \in V_\sigma$, inequality \eqref{global time-convex} was established
while deriving implication 1 $\Longrightarrow$ 2 of
Theorem~\ref{T:cross and local time-convex}.
The inequality extends to $y \in M$ by the density
of $V_\sigma$ and the continuity of $c$ on $M \times \bar M$.
\end{proof}


\begin{remark}
The density condition of Theorem~\ref{T:A3w and DASM} and Corollary~\ref{C:cross and global time-convex} can be further weakened by the works of Figalli, Loeper and Villani \cite{FV}\cite{LV} (see \cite[Theorem 12.36]{V2}).
\end{remark}

\section{Tensor products of pseudo-Riemannian metrics}\label{S:product}
This section contains the proofs of assertions 3) and 4) in Theorem~\ref{T:main}.
Assertion 3) is actually a corollary of a more general theorem,  which does not
require the transportation cost $c$ or manifold $M$ to be Riemannian:

\begin{theorem}{\bf (Products preserve non-negative cross-curvature)}\label{T:nonRiemannian products}
Let $c_\pm \in C^4(N_\pm)$ be non-degenerate non-negatively cross-curved costs
on two manifolds $N_\pm \subset M_\pm \times \bar M_\pm$.  Then
$c(x_+,x_-,\bar x_+,\bar x_-) =c_+(x_+,\bar x_+) + c_-(x_-,\bar x_-)$
is non-degenerate and non-negatively cross-curved on
$(x_+,\bar x_+,x_-, \bar x_-) \in  N_+ \times N_-$, but never {\bf A3s}.
\end{theorem}

\begin{proof}
If the cost functions $c_\pm: N_\pm \subset M_\pm \times \bar M_\pm \to \R$
define non-degenerate metrics $h_\pm$, then so does the cost
$c(x_+, x_-, \bar x_+, \bar x_-) = c_+ (x_+, \bar x_+) + c_-(x_-, \bar x_-)$ on
$N = N_+ \times N_-$, since the corresponding metric $h$ separates into block
diagonal components $h_\pm$ with non-vanishing determinants.
In this sense the geometry $(N,h)$ is the pseudo-Riemannian
tensor product of the geometries $(N_\pm, h_\pm)$.
It follows that the product $(\gamma_+(s),\gamma_-(s))$ of
geodesics $s \in [-1,1] \longrightarrow \gamma_\pm (s)$ in $N_\pm$
is a geodesic in $N= N_+ \times N_-$. Lemma \ref{L:mixed fourth derivative}
then implies
\begin{align}\label{sum of cross}
\mathop{\rm cross}(p_+ \oplus p_-, \bar p_+ \oplus \bar p_-)
= \cross_+ (p_+, \bar p_+) + \cross_-(p_-, \bar p_-)
\end{align}
for all $(x_+, \bar x_+, x_-, \bar x_-) \in N$ and
$(p_+ \oplus p_-) \oplus (\bar p_+ \oplus \bar p_-) \in T_{(x_+, \bar x_+, x_-, \bar x_-)} N$.
The proof is completed by observing that non-negativity of both summands guarantees
the same for their sum. Choosing $p_-=\bar p_+ =0$ in identity \eqref{sum of cross}
demonstrates that the tensor product cost $c$ cannot satisfy {\bf A3s}.
\end{proof}

\begin{proof}[\bf Proof of Theorem \ref{T:main}(3)]
Given two Riemannian manifolds $M_\pm = \bar M_\pm$, denote their geodesic distances
squared by $c_\pm(x_\pm,\bar x_\pm) = \frac{1}{2}\dist^2_\pm(x_\pm,\bar x_\pm)$ respectively,
and let $c(x_+,x_-,\bar x_+,\bar x_-) = \frac{1}{2}\dist^2((x_+,x_-),(\bar x_+,\bar x_-))$
denote the geodesic distance squared on $M_+ \times M_-$ equipped with the
Riemannian product metric.
Then $c(x_+, x_-, \bar x_+, \bar x_-) = c_+ (x_+, \bar x_+) + c_-(x_-, \bar x_-)$,
and the non-negative cross-curvature of $c$ on $N=N_+ \times N_-$ follows
from Theorem \ref{T:nonRiemannian products},
taking $N_\pm$ to be the domains where $c_\pm$ are smooth.
This concludes the proof of Theorem \ref{T:main}(3).
\end{proof}

The assertion 4 in Theorem~\ref{T:main} follows easily from the next lemma.
\begin{lemma}[\bf Certain cross-curvatures vanish in any Riemannian setting]
\label{L:zero cross-curvature}
Define the pseudo-metric \eqref{metric h} using the cost $c=d^2/2$ on a Riemannian
manifold $(M,d)$. Each point $(x, \bar x) \in N=M \times M \setminus \cutlocus$
with $x\ne \bar x$ admits a $2$-plane whose cross-curvature vanishes. For example,
letting $t \in [0,1] \to \gamma(t) \in M$ be a geodesic from $x=\gamma(0)$
to  $\bar x = \gamma(1)$ yields
$ 
\cross(\dot \gamma (0), \dot \gamma (1)) = 0
$ 
but
$$
h(\dot \gamma (0) \oplus \dot \gamma (1),  \dot \gamma (0) \oplus \dot \gamma (1)) = d(x,\bar x)^2.
$$
\end{lemma}

\begin{proof}
Any (affinely parameterized) Riemannian geodesic satisfies
$d(\gamma(s),\gamma(t)) =  |s-t| d(\gamma(0),\gamma(1))$.
Defining
$
f(s,t) = c(\gamma(s),\gamma(t)) = |s-t|^2 d(x,\bar x)^2/2
$
allows us to compute
\begin{equation}\label{geodesicIP}
h(\dot \gamma (0) \oplus \dot \gamma (1),  \dot \gamma (0) \oplus \dot \gamma (1))
= - \frac{\partial^2 f}{\partial s \partial t} \bigg|_{(s,t) = (0,1)}
\end{equation}
immediately.  Moreover,  the fact that $t \in [0,1] \longrightarrow (x,\gamma(t))$
is an $h$-geodesic yields the vanishing of
$\cross(\dot \gamma (0), \dot \gamma (1)) = -2\partial^4 f/\partial s^2 \partial t^2|_{(s,t)=(0,1)}$
via Lemma~\ref{L:mixed fourth derivative}.
\end{proof}

\begin{proof}[\bf Proof of Theorem \ref{T:main}(4)]
Without loss of generality, assume $M_+$ fails to be non-negatively cross-curved.
Then $\cross_+(p_+, \bar p_+) < 0$ for some
$(p_+, \bar p_+) \in T_{(x_+, \bar x_+)} N_+ $. Noting
\begin{align*}
h_+(p_+ \oplus \pm \bar p_+,p_+ \oplus \pm \bar p_+ ) & = \pm h_+(p_+ \oplus \bar p_+,p_+ \oplus  \bar p_+ )
\end{align*}
but
$\cross_+(p_+, \pm \bar p_+) = \cross_+(p_+, \bar p_+)$, we may assume that
$h_+(p_+ \oplus  \bar p_+, p_+ \oplus  \bar p_+ ) \le 0$.
Pick any nontrivial geodesic $\gamma_-$ in $M_-$.
Noting
$h_-(\dot\gamma_- (0) \oplus \dot \gamma_- (1), \dot\gamma_-(0)\oplus \dot\gamma_-(1))= d^2 (\gamma (0), \gamma(1)) > 0$
from (\ref{geodesicIP}),
one can choose  $\lambda \in \mathbf{R}$ so that
\begin{align*}
&h( (p_+ \oplus \lambda \dot \gamma_- (0) ) \oplus (\bar p_+ \oplus \lambda \dot\gamma_-(1)), (p_+ \oplus \lambda \dot \gamma_- (0) ) \oplus (\bar p_+ \oplus \lambda \dot\gamma_-(1)))\\
&= h_+(p_+ \oplus \bar p_+, p_+ \oplus \bar p_+) + \lambda^2 h_-(\dot\gamma_- (0) \oplus \dot \gamma_- (1), \dot\gamma_-(0)\oplus \dot\gamma_-(1))\\
&= 0.
\end{align*}
However, from \eqref{sum of cross} and Lemma~\ref{L:zero cross-curvature},
\begin{align*}
&\cross(p_+ \oplus \lambda \dot \gamma_- (0), \bar p_+ \oplus \lambda \dot\gamma_-(1))\\
&= \cross_+(p_+, \bar p_+) + \lambda^4 \cross_-(\dot\gamma_- (0), \dot \gamma_-(1)) \\
& = \cross_+(p_+, \bar p_+ ) < 0.
\end{align*}
\noindent This completes the proof that $(N_+ \times N_-,h_+ \oplus h_-)$
fails to be weakly regular {\bf A3w} unless both $(N_\pm,h_\pm)$ are non-negatively
cross-curved.
\end{proof}

\noindent At present, the authors do not know any Riemannian manifold which is {\bf A3w}
yet fails to be non-negatively cross-curved.
On the other hand, the proof of Theorem~\ref{T:main} (4) adapts to
non-Riemannian cost functions as in the next example, where it shows
that the tensor product of two (or more) costs on manifolds
which are not all non-negatively cross-curved is surely not {\bf A3w}.


\begin{example}[\bf A3s $\times$ A3s $\nRightarrow$ A3w]\label{E:A3w product}
Let $c_\pm(x,\bar x) = -\log|x-\bar x|$ on $M_\pm \times M_\pm \setminus \Delta$,
$M_\pm=\R^n$, $\Delta:=\{(x, x) \ | \ (x, x) \in M_\pm \times M_\pm \}$.
This logarithmic cost function is known to be {\bf A3s}
\cite{MTW} \cite {TW} but it does not induce non-negatively cross-curved pseudo-metric as
indicated, e.g.\ in \cite[Example 3.5]{KM2}. To see this fact one
uses \eqref{Alternative A3}, whose righthand-side coincides for this logarithmic cost with
\begin{align*}
p^i p^j f_{ij} \Big{|}_{(Df)^{-1} (-q_0^*-q^*)}
= 2 ((q_0+q)^*_i p^i)^2 - |p|^2 |(q_0+q)^*|^2
\end{align*}
where $f(x-\bar x) := -\log|x- \bar x|$. Here the righthand-side is strictly convex
with respect to $q^*$ along the Euclidean line parallel to $p$, but strictly concave
along the nullspace of $p$.
As a result  $\cross_\pm(p_\pm, \lambda \bar p_\pm ) < 0$ for
$p_\pm$ parallel to $\bar D D c(x,\bar x) \bar p_\pm$ as Euclidean vectors
and $\lambda\ne0$. On the other hand,
$$
h_+(p_+ \oplus \bar p_+,p_+ \oplus \bar p_+) +
h_-(p_- \oplus \lambda \bar p_-, p_- \oplus \lambda \bar p_-) \\
$$
vanishes for some non-zero $\lambda \in \R$.
From \eqref{sum of cross}, the pseudo-Riemaniann metric $h$ induced by
the cost $c( (x,\bar x), (y, \bar y) )= c_+(x_+, \bar x_+) + c_- (x_-, \bar x_-)$
on $(M_+ \times M_- ) \times (M_+ \times M_-)$ then fails to be {\bf A3w}. 
\end{example}


\begin{remark}[\bf tensor product examples, Loeper's maximum principle, and
time-convexity of the sliding mountain] \label{R:products and DASM}

As  mentioned in the introduction, Loeper's maximum principle ({\bf DASM})
is a key property for the regularity of optimal maps. The conclusions of
Theorem~\ref{T:A3w and DASM} and Corollary~\ref{C:cross and global time-convex}
(hence {\bf DASM}) hold for the distance squared cost on the
Riemannian product of round spheres
$M= \bar M= {\mathbf S}^{n_1} \times \cdots \times {\mathbf S}^{n_k}\times \R^l$ --- thus also on its Riemannian
submersions (see Theorem~\ref{T:time-convex and submerstion}).  To see this, first note that by the result of present section and Section~\ref{S:sphere cross-curv}, $M$ satisfies non-negative cross-curvature on $N = M \times \bar M \setminus \cutlocus$.
The density condition of $\cap_{0 \le t \le 1} V(x, \bar x(t))$ is easily checked since the cut
locus of one point in this example is a smooth submanifold of codimension greater than or
equal to $2$. This new global result illustrates an advantage of our method over
other approaches \cite{TW} \cite{TW2} \cite{L},
where one would require a regularity result for optimal maps (or some a priori estimates)
to obtain the conclusion of Theorem~\ref{T:A3w and DASM}.
For example, to implement these other
approaches for the manifolds of this tensor product example, one would need to establish that an optimal map
remains \emph{uniformly away from the cut locus}, as is currently known only
for a single sphere $M = \bar M = {\mathbf S}^n$ from work of
Delan\"oe \& Loeper \cite{DL}
(alternately \cite{L} or Appendix of \cite{KM2}), for the case of perturbations
\cite{DG}\cite{LV}\cite{V3},
and for some Riemannian submersion quotients (see Section~\ref{S:regularity and submersion})
of $\mathbf{S}^n$, including the complex projective spaces ${\bf CP}^n$ with Fubini-Study metric.
To the best of our knowledge, no one has yet succeeded in establishing regularity results
for this tensor product example, though as mentioned above
we hope to resolve this in a subsequent work.
\end{remark}

\section{Riemannian submersions and cross-curvature}\label{S:submersion}
In this section we prove the assertion 2) in Theorem~\ref{T:main}.
The key result is an O'Neill type inequality contained in Theorem~\ref{T:cross and submersion}
which compares cross-curvatures in Riemannian submersion.
Theorem~\ref{T:time-convex and submerstion} deals with the survival of global properties
such as Loeper's maximum principle ({\bf DASM}) and time-convexity of the sliding mountain,
under  Riemannian submersion.
These results are applied to show the regularity of optimal maps on the Riemannian submersion quotients of the round sphere, for example, the complex projective space ${\bf CP}^n$ with Fubini-Study metric (see Section~\ref{S:regularity and submersion}). From now on we focus on the case of Riemannian manifolds
(with $c=\frac{1}{2}\dist^2$). In this case  the $c$-exponential map coincides with the
Riemannian exponential map: $\cexp = \exp$.

Recall the definition and basic facts of Riemannian submersion.
\begin{definition}[\bf Riemannian submersion]\label{D:submersion} (See \cite{CE}.)
A surjective differentiable map $\pi: M \to B$ from a Riemannian manifold $M$ onto a Riemannian manifold $B$ is said to be a \emph{Riemannian submersion} if the following hold:
\begin{itemize}
\item $\pi$ is a submersion, i.e. $d \pi: T_x M \to T_{\pi (x)} B$ is surjective
for each $x\in M$;
\item for the orthogonal decomposition $T_x M = \ker d \pi  \oplus (\ker d \pi )^\bot$ for each $x \in M$, $d \pi \big{|}_{(\ker d \pi )^\bot}$ is an isometry.
\end{itemize}
The subspaces $\mathcal{V}:=\ker d \pi$, $\mathcal{H}:=(\ker d \pi)^\bot$ are called
\emph{vertical} and \emph{horizontal} subspaces, respectively.
For each $ v \in T_{b} B$, $b \in B$, there exists its unique \emph{horizontal lift} $\tilde v \in T_x M \cap \mathcal{H}$ for each $x\in \pi^{-1}(b)$ such that $d\pi (\tilde v) = v$.
We use the metric identifications $T_b^* B = T_b B$ and $T_x^*M = T_x M$
to extend the definition of a horizontal lift to cotangent vectors.
For each piecewise smooth curve $\gamma: [0,1] \to B$ and $x \in \pi^{-1}(\gamma (0))$, there exists its \emph{horizontal lift} $\tilde \gamma: [0,1] \to M$ such that
$\gamma(0)= x$, $\pi (\tilde \gamma) = \gamma$, $\dot {\tilde \gamma} (t)  \in \mathcal{H}$. Moreover, $\gamma$ is a  geodesic if and only if its horizontal lift $\tilde \gamma$ is. If $\gamma$ is minimal, then so is $\tilde{\gamma}$.
This property  yields, for the horizontal lifts $\tilde v \in T_{ x} M \cap \mathcal{H}$ and
$ x \in \pi^{-1} (b)$ of  $v \in T_b B$ and $b \in B$,
\begin{align}\label{exp lift}
\pi(\exp_{ x} \tilde v) = \exp_b v .
\end{align}
Regarding Riemannian distance,
\begin{align}\label{distance and submersion}
\hbox{\rm dist}_M (x, y) \ge \hbox{\rm dist}_B (\pi(x), \pi (y)).
\end{align}
We call $M$ the \emph{total space} of the Riemannian submersion and $B$ the \emph{base} of the Riemannian submersion or the \emph{Riemannian submersion quotient}.
\end{definition}
\noindent Many examples of Riemannian submersions may be found in Cheeger \& Ebin \cite{CE}
or Bess\cite{Be}.
Every Riemannian covering projection is obviously a Riemannian submersion.
Other important examples are Hopf fibrations such as complex and quaternionic
projective spaces ${\bf CP}^m$ and ${\bf HP}^m$,  where the standard round sphere $S^n$
(sectional curvature $\equiv 1$) is the total space.
\begin{example}[\bf Hopf fibrations]\label{E:Hopf fibration}(See \cite{Be} pages 257--258.)
$\pi: S^{2m+1} \to {\bf C P}^m$,
$\pi: S^{4m+3} \to {\bf HP}^m$.
The base spaces ${\bf CP}^m$ and ${\bf HP}^m$ have real dimensions $2m$, $4m$,
respectively, and have non-isotropic sectional curvatures $K$, $1 \le K \le 4$.
\end{example}

Our goal in this section is to compare the cross-curvature of the base space with that of
the total space of Riemannian submersion. For this purpose, we use the definition above
to assign to each pair of points in the base space a corresponding pair of points
horizontally lifted to the total space.

\begin{definition}[\bf horizontal lift of a pair of points]\label{D:horizontal lift of a pair}
Let $\pi : M \to B$ be a Riemannian submersion.
For each pair of points $(x, \bar x) \in N_B = B \times B \setminus \cutlocus$, the unique geodesic $\gamma: [0,1] \to B$ with $\gamma(0)= x$, $\gamma (1) = \bar x$ has
its horizontal lift $\tilde \gamma$ with $\tilde \gamma(0)= \tilde x$, $\tilde\gamma(1)=\tilde{\bar x}$.
Then the pair $(\tilde x, \tilde{\bar x}) \in N_M = M \times M \setminus \cutlocus$ is said to be a \emph{horizontal lift of $(x, \bar x)$}.
\end{definition}

To have a bit more general conclusion in the next theorem, we observe as in \cite{M} 
\begin{lemma}[\bf radial cost-exponential][see \cite[Theorem 13]{M}]\label{L:radial cost-exponential}
Let $f: \R \to \R$ be a  strictly convex smooth function, thus the derivative $f'$ has an inverse function. Further assume that $f$ is strictly increasing on $\R_+$.
Let $M = \bar M$ be a Riemannian manifold, and $c(x, \bar x) = f(\mathop{\rm dist}(x,\bar x))$
a cost function defined on $M \times \bar M$, where $\mathop{\rm dist}$ denotes the
Riemannian distance. The $\cexp$ and $\exp$ have the relation
\begin{align}\label{radial c-exp}
\cexp_x (p^*) = \left\{ %
\begin{array}{ll}\exp_x \big{(} \frac{(f')^{-1} (|p|)}{|p|} p \big{)} & \hbox{ \ for  $0 \ne p^* \in \dom \cexp_x$},\\
 x & \hbox{ \ for $p^* = 0 $}
 \end{array}%
 \right.
 \end{align}
where the vector $p$ and the co-vector $p^*$ are identified by the Riemannian metric.
This cost $c$ induces a pseudo-Riemannian metric
$h_M$ on $N_M := M \times M \setminus \cutlocus$ as in \eqref{metric h}
which is {\bf A2} non-degenerate and {\bf A1} bi-twisted.
\end{lemma}

\begin{proof}
For $f(\rho) = \rho^2$, this fact is well-known. For general $f$,
it follows from
\begin{align*}
-D c (x, \bar x) = -\frac{f'(\mathop{\rm dist}(x, \bar x))}{\mathop{\rm dist}(x, \bar x)} D (\frac{\hbox{\rm dist}^2}{2})(x, \bar x),
\end{align*}
where the inverse $(f')^{-1}$ exists for strictly convex function $f$.
The last assertion can be checked directly from Definition~\ref{D:A1,2}. This completes the proof.
\end{proof}

Now we are ready to show our main theorem in this section.
\begin{theorem}[\bf cross-curvature and Riemannian submersion]\label{T:cross and submersion}
Let $\pi: M \to  B$ be a Riemannian submersion from $M$ to $B$.
Let $f: \R \to \R$ be a strictly convex smooth function that is strictly increasing on $\R_+$, 
and $c_M = f \circ {\mathop{\rm dist}}_M$ and $c_B=f\circ {\mathop{\rm dist}}_B$
be cost functions defined on $M \times M$ and $B \times B$,
where $\mathop{\rm dist}$ denotes the Riemannian distance of the corresponding
manifold; $c_M$ and $c_B$
induce the pseudo-Riemannian metrics $h_M$ and $h_B$
on $N_M := M \times M \setminus \cutlocus$ and
$N_B:=B \times B \setminus \cutlocus$ respectively, as in \eqref{metric h}.
Fix $(x, \bar x) \in N_B$ and
let $(\tilde x, \tilde{\bar x}) \in N_M$ be a horizontal lift of $(x, \bar x)$.
Given $v \oplus \bar v \in T_{(x,\bar x)} N_B$ there exists
$\tilde w \oplus \tilde {\bar w} \in T_{(\tilde x, \tilde{\bar x})} N_M$
with $d\pi_{\tilde x} (\tilde w) = v$, $d\pi_{\tilde{\bar x}} (\tilde{\bar w}) = \bar v$,
 such that
\begin{equation}\label{metric lift}
h_B(v \oplus \bar v , v \oplus \bar v)
= h_M(\tilde w \oplus \tilde{\bar w},  \tilde w \oplus \tilde{\bar w})
\end{equation}
and
\begin{equation}\label{O'Neill}
\hbox{\rm cross}^{(N_B, h_B)}_{(x, \bar x)} (v, \bar v)
 \ge \hbox{\rm cross}_{(\tilde x, \tilde {\bar x})}^{(N_M, h_M)}
(\tilde{w}, \tilde {\bar w} ).
\end{equation}
For example, it suffices to take $\tilde w^*=-\bar D D c_M(\tilde x,\tilde {\bar x}) \tilde {\bar w} \in T^*_{\tilde x}M$
to be the horizontal lift of $v^*=-\bar D D c_B(x,{\bar x}) {\bar v} \in T^*_{x}B$
and $\tilde{\bar w}^*= -D \bar D c_M(\tilde x, \tilde {\bar x}) \tilde w \in T^*_{\tilde {\bar x}}M$
to be the horizontal lift of $\bar v^*=-D \bar D c_B(x, {\bar x}) v \in T^*_{\bar x}B$.
\end{theorem}

\begin{proof}[Proof of \eqref{O'Neill}]
Let $(\tilde x,\tilde{\bar x})$ be a horizontal lift
of $(x,\bar x) \in N_B$, and define $q^*$ and $\bar q^*$ by
$\cBexp_x q^* = \bar x$ and $\cBexp_{\bar x} \bar q^* = x$.
Then $\dist_B(x,\bar x) = \dist_M (\tilde x,\tilde{\bar x})$ and it follows from
(\ref{radial c-exp}) that
the horizontal lifts $\tilde q^*$ of $q^*$ and $\tilde {\bar q}^*$ of $\bar q^*$
satisfy $\cMexp_{\tilde x} \tilde q^* = \tilde{\bar x}$ and
$\cMexp_{\tilde{\bar x}} \tilde{\bar q}^*= \tilde x$.
To fixed $v \in T_x B$ and $\bar v \in T_{\bar x}B$
correspond $\bar v^* \in T^*_{\bar x} B$ and $v^* \in T^*_x B$ such that
$\bar v^*= -D \bar D c_B(x, {\bar x}) v$ and 
$v^*= - \bar D D c_B(x,{\bar x}) {\bar v}$.
Equivalently, $v^* \oplus \bar v^* = 2h_B(v \oplus \bar v, \;\cdot\;)$.

Now let $\Sigma: (s,t) \in \mathopen[-1,1\mathclose]^2 \longrightarrow (x(s),\bar x(t)) \in N_B$
be the surface given by $x(s) = \cBexp_{\bar x}( \bar q^* + s \bar v^*)$ and
$\bar x (t)= \cBexp_x(q^* + t v^*)$. By Lemma~\ref{L:c-segments are geodesics},
the curves $\sigma(s)=(x(s), \bar x)$ and
$\tau(t) = (x,\bar x(t))$ through 
$(x(0),\bar x (0))= (x, \bar x)$ are $h_B$-geodesics.
Lift $\Sigma$ to
$\tilde \Sigma: (s,t) \in \mathopen[-1,1\mathclose]^2 \longrightarrow (\tilde x(s),\tilde{\bar x}(t)) \in N_M$
by setting
$\tilde x (s) = \cMexp_{\tilde{\bar x}} (\tilde{\bar q}^* + s \tilde{\bar w}^*)$,
and $\tilde{\bar x}(t)=\cMexp_{\tilde x} (\tilde q^* + t \tilde w^*)$, where
$\tilde {\bar w}^*$ and
$\tilde {w}^*$ are the horizontal lifts of $\bar v^*$ and $v^*$ respectively.
Thus,
$\tilde \sigma (s)=(\tilde x (s), \tilde{\bar x})$ and $\tilde \tau (t) =(\tilde x, \tilde{\bar x}(t))$ are $h_M$-geodesics, with
$\tilde \sigma (0)=(\tilde x, \tilde{\bar x}) = \tilde \tau (0)$.
Moreover, $\pi (\tilde x (s)) = x(s)$ and $\pi(\tilde{\bar x}(t)) = \bar x (t)$
from \eqref{radial c-exp} and \eqref{exp lift}, so taking $\tilde w := \dot {\tilde x}(0)$ and
$\tilde {\bar w} := \dot{ \tilde {\bar x}}(0)$ yields
$d\pi_{\tilde x} (\tilde w) = \dot x(0)=v$ and
$d\pi_{\tilde {\bar x}}(\tilde {\bar w}) = \dot {\bar x}(0)=\bar v$.
Notice that
\begin{align*}
-D \bar D c_B (x(0),\bar x) \dot x (0) &= \bar v^*, \qquad -\bar D D c_B(x,\bar x(0)) \dot{\bar x}(0) = v^* \ ;\\
-D \bar D c_M (\tilde x(0),\tilde {\bar x }) \dot{\tilde x}(0) &= \tilde{\bar w}^*, \qquad  -\bar D D c_M(\tilde x, \tilde {\bar x}(0))\dot{\tilde{\bar x}}(0)= \tilde w^*.
\end{align*}
Define an auxiliary function
\begin{align*}
F(s, t) :=  c_M (\tilde x (s), \tilde{\bar x} (t))-c_B(x(s), \bar x(t)).
\end{align*}
From Lemma~\ref{L:mixed fourth derivative}, the inequality $\frac{\partial^4}{\partial s ^2 \partial t^2} F(s, t) \big{|}_{(0,0)} \ge 0$
will imply (\ref{O'Neill}).
First, observe from \eqref{distance and submersion} and the monotonicity of $f$ that $c_B(x(s), \bar x (t)) \le c_M (\tilde x (s), \tilde{\bar x}(t))$, thus,
\begin{equation}\label{nonnegative f(s,t)}
F(s,t) \ge 0 \quad {\rm for} \quad (s, t) \in [-1,1]^2.
\end{equation}
We shall verify the desired inequalities by computing the Taylor expansion of
$F$ at $(0,0)$ to fourth order.
First observe that $F(s,0) = 0 = F(0,t)$ for all $|s|,|t| \le 1$, so
$$F(s,t) = f_{11} s t + f_{21} s^2 t + f_{12}s t^2 + f_{31} s^3 t + f_{22} s^2 t^2 + f_{13}s t^3
+ O(|s| + |t|)^{5}
$$
as $|s| + |t| \to 0$.
Since $F(s,\pm s) \ge 0$ for $|s| \le 1$, we deduce the
vanishing of $f_{11},f_{12}$ and $f_{21}$ in turn, and the
inequalities $f_{31} + f_{22} + f_{13} \ge 0$ and $f_{22} - f_{31} - f_{13} \ge 0$.
This implies $f_{22} \ge 0$ as desired.
(Although not needed here, $f_{31}=0$ follows from $F(s,s^2) \ge 0$, and $f_{13}$ vanishes
similarly.) Noting
\begin{align*}
f_{11}
&= \frac{\partial^2 F}{\partial s \partial t}(0,0) \\
&= \dot {\tilde x}(0) \bar D D c_M (\tilde x, \tilde {\bar x}) \dot {\tilde {\bar x}}(0)
 - \dot        x(0) \bar D D c_B (x,\bar x) \dot        {\bar x}(0) \\
&= - h_M(\dot {\tilde x}(0) \oplus \dot {\tilde {\bar x}}(0), \dot {\tilde x}(0) \oplus \dot {\tilde {\bar x}}(0))
   +  h_B(\dot        x(0) \oplus \dot        {\bar x}(0), \dot        x(0) \oplus \dot        {\bar x}(0))
\end{align*}
we have established (\ref{metric lift}) en passant
to complete the proof.
\end{proof}

Before stating a corollary of this theorem, let's make a provisional definition
which can serve as a strict cross-curvature condition for a Riemannian manifold.
Notice from Lemma~\ref{L:zero cross-curvature} that for each pair of points in a
Riemannian manifold there are tangent vectors with zero cross-curvature.

\begin{definition}[\bf almost positive cross-curvature]\label{D:almost cross-curvature}
A Riemannian manifold $M$ with positive sectional curvature
is said to be \emph{almost positively cross-curved}
if  for each $(x, \bar x) \in N=M \times M \setminus \cutlocus$ such that $x \ne \bar x$ and $p \oplus \bar p \in T_{(x, \bar x)} N$,
\begin{equation}\label{almost positive}
\cross(p, \bar p) \ge 0
\end{equation}
and the equality holds if and only if $p$ and $\bar p$ are parallel to the velocity vectors
$\dot \gamma(0)$ and $\dot \gamma(1)$, respectively, for the unique geodesic
$t \in [0,1] \to \gamma(t) \in M$ from $x$ to $\bar x$.
\end{definition}

\noindent For example, the standard round sphere is almost positively cross-curved as shown in Section~\ref{S:sphere cross-curv}.

\begin{corollary}[\bf {\bf A3w/\bf A3s}, non-negative/almost positive cross-curvature survive Riemannian submersion]\label{C:cross and submersion} Let $\pi: M \to B$ be a Riemannian submersion.
If the distance squared cost on $M$ satisfies {\bf A3w}, {\bf A3s}, non-negative cross-curvature, or almost positive cross-curvature condition, then
$B$ satisfies the same condition, respectively.
\end{corollary}
\begin{proof}
The relevant inequalities for the cross-curvature follow directly from \eqref{O'Neill}.
Let's consider especially  the equality case of almost positive cross-curvature.
Assume $M$ is almost positively cross-curved.  Suppose
$$\cross_{(x, \bar x)}^{(N_B, h_B)}(p, \bar p) = 0$$ for $x \neq \bar x$, $(x, \bar x) \in N_B$,
$p\oplus \bar p \in T_{(x,\bar x)} N_B$.
Lift the unique minimizing geodesic $\gamma$ linking $\gamma(0)=x$ to $\gamma(1) = \bar x$
to a horizontal geodesic $\tilde \gamma$ joining $\tilde x = \tilde \gamma(0)$
to $\tilde {\bar x} = \tilde \gamma(1)$.
There is a unique choice $\tilde p \oplus \tilde {\bar p} \in T_{(\tilde x, \tilde{\bar x})}N_M$
such that each component $\tilde p^*$ 
and $\tilde {\bar p}^*$ 
of $\tilde p^* \oplus \tilde {\bar p}^* = h_M(\tilde p \oplus \tilde {\bar p}, \cdot)$
is the horizontal lift of the corresponding component of
$p^* \oplus \bar p^* = h_B(p \oplus \bar p, \cdot)$.
The preceding theorem asserts
$d \pi_{\tilde x} (\tilde p) = p$ and $d \pi_{\tilde {\bar x}}(\tilde {\bar p}) = \bar p$.
Apply \eqref{O'Neill} to get $\cross_{(\tilde x, \tilde{\bar x})}^{(N_M, h_M)}(\tilde p, \tilde{\bar p}) = 0$.
Then by almost positive cross-curvedness of $M$, the vectors $\tilde p$ and $\tilde{\bar p}$ are
parallel to the velocities $\dot{\tilde \gamma} (0)$ and $\dot{\tilde \gamma}(1)$, respectively.
This implies that the vector projections $p$ and $\bar p$ are parallel to
$\dot \gamma (0) = d\pi_{\tilde x}(\dot {\bar \gamma}(0))$ and
$\dot \gamma (1) = d\pi_{\tilde {\bar x}}(\dot {\bar \gamma}(1))$, respectively,
to complete the proof.
\end{proof}

Let's now turn to more global aspects of the distance squared cost function
under Riemannian submersion.
Though {\bf local DASM}/{\bf local time-convex sliding mountain} are equivalent
to {\bf A3w}/nonnegative cross-curvature, respectively (see Theorem~\ref{T:cross and DASM} and Theorem~\ref{T:cross and local time-convex}), their global counterparts  {\bf DASM/time-convex sliding mountain} require additional conditions on the geometry of the domain (see Theorem~\ref{T:A3w and DASM} and Corollary~\ref{C:cross and global time-convex}). The following theorem,
uses a simple comparison of distance to give a direct proof that both
Loeper's maximum principle and time-convexity of the sliding mountain
survive Riemannian submersion even in the absence of restrictions on domain geometry.

\begin{theorem}[\bf Loeper's maximum principle and time-convexity of the sliding mountain survive Riemannian submersion]\label{T:time-convex and submerstion}
Let $\pi: M \to  B$ be a Riemannian submersion.
Compose a smooth strictly convex function $f:\R \longrightarrow \R$ that is strictly increasing on $\R_+$, 
with the Riemannian distance on $M$ to define a cost function $c_M = f \circ \mathop{\rm dist_M}$.
It induces a pseudo-Riemannian metric
$h_M$ on $N_M := M \times M \setminus \cutlocus$ as in \eqref{metric h}
which is {\bf A2} non-degenerate and {\bf A1} bi-twisted.
Similarly $c_B = f \circ \mathop{\rm dist_B}$ defines a non-degenerate and bi-twisted
pseudo-metric $h_B$ on $N_B := B \times B \setminus \cutlocus$.
Suppose that for each $h_M$-geodesic of the form $t \in [0,1] \longrightarrow (\tilde x, \tilde {\bar x}(t)) $ in $N_M$,
 $\tilde f_t(\tilde y) = -c_M(\tilde y,\tilde {\bar x}(t)) + c(x,\tilde {\bar x}(t))$ satisfies
\eqref{DASM} (or \eqref{global time-convex} respectively) for each $\tilde y \in M$. Then for each $h_B$-geodesic $t \in [0,1] \longrightarrow (x, \bar x(t)) \in N_B$, $f_t(y) = -c_B(y,\bar x(t)) + c_B(x,\bar x(t))$ satisfies the same inequality for each $y \in B$.
\end{theorem}

\begin{proof}
Let $t \in [0,1] \longrightarrow \sigma (t)= (x, \bar x(t)) \in N_B$
be an $h_B$-geodesic and set $\bar x:=\bar x(0)$.
Define the sliding mountain $f_t(\cdot):=-c_B(\cdot, \bar x(t)) + c_B(x, \bar x(t))$ on $B$.
Identify tangent vectors with co-tangent vectors by the Riemannian metric.
By Lemma \ref{L:c-segments are geodesics}
there exist $p, q \in T_x B$ such that $\bar x(t)= \cexp_x (p + t q)$.
Lift
$p, q$ to horizontal vectors $\tilde p, \tilde q$ at $\tilde x \in \pi^{-1} (x)$.
Let $\tilde{\bar x}(t) =\cexp_{\tilde x} (\tilde p + t \tilde q)$. From the Riemannian submersion property and Lemma~\ref{L:radial cost-exponential},
\begin{align*}
\hbox{$c_B(x, \bar x(t))= (f')^{-1}(|p+tq|_B) 
=c_M (\tilde x, \tilde{\bar x}(t))$ for all $t \in [0,1]$;}
\end{align*}
and  $(\tilde x,\tilde{\bar x} (t)) \in \bar N_M$ for $t \in [0,1]$. This last property
comes from the fact that Riemannian submersions lift the minimal geodesic from $x$ to $\bar x (t)$ to the  minimal geodesic from $\tilde x$ to $\tilde{ \bar{ x}}(t)$.  Thus,
$t\in [0,1] \to (\tilde x, \tilde{\bar x}(t)) \in N_M$ gives an $h_M$-geodesic.
Define
$\tilde f _t (\cdot)= -\tilde c ( \cdot, \tilde{\bar x} (t)) + c(\tilde x, \tilde{\bar x} (t))$.
Fix $t \in [0,1]$, $y \in B$. Let $\gamma:[0,1] \to B$ be a geodesic from $\gamma(0)= \bar x(t)$ to $\gamma(1)=y$. Let  $\tilde \gamma$ be the horizontal lift of $\gamma$ such that $\tilde{\gamma}(0)= \tilde{\bar x}(t)$. Let $\tilde y := \tilde{\gamma}(1) \in \pi^{-1}(y)$. Notice that
\begin{align*}
\hbox{$c_B(y, \bar x(t))=c_M(\tilde y, \tilde{\bar x}(t))$; \ \ \
$c_B(y, \bar x (s)) \le c_M (\tilde y, \tilde{\bar x}(s))$ for all $s \in [0,1]$.}
\end{align*}
The last inequality is from \eqref{distance and submersion} and the monotonicity of $f$.
Therefore,
\begin{align}\label{comparison fB fM}
f_t (y) = \tilde f_t (\tilde y) \ ;\ \ \ \tilde f_s (\tilde y) \le f_s (y)  \hbox{ for all $s \in [0,1]$}.
\end{align}
Now, assume $s \in [0,1] \longrightarrow \tilde f_s(\tilde y)$ is convex.
Choosing $s=0,1$,
from \eqref{comparison fB fM},
\begin{align*}
f_t (y) = \tilde f_t (\tilde y)
&\le (1-t) \tilde f_0 (\tilde y) + t \tilde f_1 (\tilde y)\\
& \le  (1-t) f_0 ( y) + t  f_1 ( y).
\end{align*}
Since $t \in [0,1]$ was arbitrary,  the same convexity holds for
$t\in[0,1] \longrightarrow f_t(y)$.
The survival of Loeper's maximum principle ({\bf DASM}) follows by a similar argument.
\end{proof}
\section{Regularity of optimal maps on the complex projective space ${\bf CP}^n$}\label{S:regularity and submersion}
In this section, we establish regularity of optimal maps on the complex projective space ${\bf CP}^n$, in fact more generally on the Riemannian submersion quotient of the round sphere with strictly convex domain of exponential map.

From Corollary~
\ref{C:cross and submersion} and Theorem~\ref{T:time-convex and submerstion}, any Riemannian
submersion quotient $B$ of the round sphere satisfies, in particular, {\bf A3s} and
Loeper's maximum principle ({\bf DASM}).
Let's further assume that the domain of exponential map $\dom(\text{Exp}_{x} )$
is strictly convex for each $x \in B$. Here the domain of exponential map is the
special case of \eqref{dom c-exp} with $c= \frac{1}{2}\dist^2$. We also identify vectors and co-vectors by the Riemannian metric. The Riemannian
manifolds ${\bf CP}^n$ and ${\bf HP}^n$ in Example~\ref{E:Hopf fibration} all
satisfy this condition.
In these cases, the domain of the exponential map is the ball of radius $\frac{\pi}{2}$ in the tangent space,
and the conjugate locus coincides with the cut locus.  Thus these manifolds
are focal, in contrast to the nonfocal manifolds analyzed by Loeper and Villani in the preprint version of \cite{LV}.\footnote{Remark added in revision:
  Shortly after we communicated the present manuscript to Loeper and Villani,  we learned they had revised \cite{LV} to address
{\bf A3s} manifolds which are not purely-focal.
Our results of the preceding section establish
${\bf CP}^n$ and ${\bf HP}^n$ to be {\bf A3s}, hence examples of such manifolds.  Combining their revision with Theorem \ref{T:cross and submersion},  one may obtain regularity of
optimal transportation on any Riemannian submersion of the round sphere simply by showing 
that no minimizing geodesic on the submersed manifold is purely focal.}
Exactly the same method as in Appendix E of \cite{KM2},
which uses Loeper's argument \cite{L}, shows the H\"older continuity of optimal maps
for the source and target measures $\rho$, $\bar \rho$ on $B$, assuming there exists
$\lambda>0$ such that
$$
\lambda \le \frac{d\bar \rho}{d\vol} \quad {\rm and} \quad \frac{d\rho}{d\vol} \le 1/\lambda,
$$
or any of the weaker hypotheses proposed by Loeper \cite{L}.
Notice that the distance function is not differentiable
anywhere on the cut-locus (at least in ${\bf CP}^n$ and ${\bf HP}^n$).
Otherwise, there is a nonzero gradient $\nabla_y \dist (x_0, y)$ of the distance function at $y$ in the cut-locus of $x_0$. Since the domain of exponential is the round ball, moving from $y$ along the direction $\nabla_y \dist (x_0, y)$ should decrease the distance, a contradiction. Therefore, the continuity result we obtained also shows that optimal maps stay uniformly away from the cut-locus, which
enables one to apply the  method of Ma, Trudinger and Wang \cite{MTW} as done by
Loeper \cite{L2} to get higher regularity of optimal maps for smoother source and target
measures.

It is natural to expect,  but unknown to us, whether
every Riemannian submersion quotient of the round sphere satisfies the strict convexity
of the domain of exponential map which is the only additional condition required for the method of the preceding
paragraph  to apply (at least for continuity of optimal maps).

For completeness we mention that for the case of a covering map (Riemannian submersion with
discrete fibers) of the round sphere, it is known that lifting of the measures on $B$ to
the total space $M$ can be applied to show the regularity of optimal maps using
established regularity results \cite{L2} on the round sphere. This was discovered
independently by Delano\"e and Ge in Appendix C of \cite{DG}.
An alternative approach (in the same spirit to our discussion above) to this covering case has also been
given by Figalli and Rifford \cite{FR}.


\section{Sphere is almost positively cross-curved}\label{S:sphere cross-curv}
In this section we show our final result, namely
that the standard round sphere is almost positively cross-curved.
This represents a significant advance over Loeper's discovery that the
round sphere satisfies {\bf A3s}. Its proof will require the following elementary lemma.

\begin{lemma}[Calculus fact]\label{the function}
For $0 \le \rho \le \pi$, the function
\begin{align*}
a(\rho):=\sin^2\rho+ \rho \sin \rho -\rho^2 (1+ \cos \rho)
\end{align*}
satisfies $a(\rho)  \ge 0$. Moreover $a(\rho) = 0$ if and only if $\rho = 0, \pi$.
\end{lemma}
\begin{proof}
Reparameterize $\rho:=\pi/2 + \arcsin (\lambda)$ by $|\lambda| < 1$.
Then
\begin{align*}
& a(\pi/2 + \arcsin(\lambda))\\
& = 1-\lambda^2  + (\pi/2 + \arcsin(\lambda))\sqrt{1-\lambda^2}-(\pi/2 + \arcsin(\lambda))^2 (1-\lambda).
\end{align*}
Define
\begin{align*}
b(\lambda):= \frac{a(\pi/2 + \arcsin(\lambda))}{1-\lambda}.
\end{align*}
The assertion holds if $b(\lambda) > 0$, for $|\lambda|<1$.
From
$$
(1-\lambda) b'(\lambda) = 2-\lambda + (\frac{\pi}{2} + \arcsin \lambda)(2\lambda -1) /(1-\lambda^2)^{1/2}
$$
one can check $b(-1)=0$, $b'(-1)=0$, and $b'(\lambda) \ge 3$ if $\frac{1}{2} \le \lambda \le 1$.
Moreover,
$$\sqrt{1-\lambda^2}(1-2\lambda)^2 \frac{d}{d\lambda}\Big{(}\frac{(1-\lambda)\sqrt{1-\lambda^2}}{1-2\lambda}\frac{d b}{d\lambda} \Big{)}\\
=2(1-\lambda^2)(1+\lambda) > 0
$$
for $-1 \le \lambda < \frac{1}{2}$,  which shows $b'(\lambda)$ increases monotonically
in this range. Thus $b'(\lambda)$ and $b(\lambda)$ both remain positive throughout
$|\lambda|<1$, completing the proof.
\end{proof}

\begin{theorem}[\bf Sphere is almost positively cross-curved.]\label{T: sphere cross}
The $n$-dimensional sphere $M=S^n$ with the standard round metric (i.e., sectional curvature $K \equiv 1$) is almost positively cross-curved
\eqref{almost positive}, a fortiori non-negatively cross-curved.
%
\end{theorem}

\begin{proof}
This theorem follows from the following nontrivial (and tedious) calculations.
Let's first set up the geometric configuration we are going to analyze.
Let $x$ be a point in the round sphere $S^n$ of diameter $\pi$.
Fix two unit tangent vectors $q, w \in T_x S^n$, $|q|,|w| > 0$.
For $t \in \R$, with $|t|$ sufficiently small, let $r(t)$ be a line in $T_x S^n$ with $\dot r (t) = q$, $\ddot r (t) =0$, $ | r(t)| < \pi$, where
$\dot f$, $\ddot f$ denote the time derivatives $\frac{d}{dt}f (t)$, $\frac{d^2}{dt^2} f(t)$ of a function $f$. Let $\bar x (t)$ be the $c$-segment $\bar x(t):=\exp_x r(t)$. Denote $\bar x = \bar x (0)$, $\rho = |r(t)|$ (thus $ 0 \le \rho < \pi$), and $\hat r = \frac{r(t)}{|r(t)|}$. Let $\langle \cdot, \cdot \rangle $ denote the Riemannian inner product.

To apply Lemma~\ref{L:mixed fourth derivative},  define
\begin{align}\label{H-0}
H:=\mathop{\rm Hess}(\frac{\hbox{\rm dist}
 (\cdot, \bar x(t))^2}{2})\Big{|}_{x} (w, w).
\end{align}
   To prove \eqref{almost positive} we will first show $-\ddot H = -\frac{d^2}{dt^2} H \ge 0$, and then the equality case shall be determined.
   By continuity we may assume $0< \rho < \pi$ without loss of generality.
From a standard Riemannian geometry calculation (for example see \cite{DL}\cite{L2}), one can show
that
\begin{align}\label{H-special}
H=|w|^2- I \, G ,
\end{align}
where
\begin{align*}
I:=|w|^2-\langle \hat r, w \rangle^2, \  \ \ \
G:= 1-\frac{\rho \cos \rho}{\sin \rho}.
\end{align*}

\subsection*{Step 1: reduction to $2$-dimensional case}
One of the key points of the proof is to rearrange the expression of $-\ddot{H}$ in a clever way to enable further analysis.
Before differentiating $H$, let's list some preliminary computations in the order of complexity. Define a function
\begin{align*}g (u) = -\frac{u \cos u}{\sin u}, \qquad u \in ]0, \pi[.
\end{align*}
Then,
\begin{align*}
g'(u)\Big{|}_{u=\rho}= \frac{1}{\sin \rho}  B , & \qquad g''(u)\Big{|}_{u=\rho} = \frac{\rho}{\sin^3 \rho}A.
\end{align*}
where
\begin{align*}
A:=\frac{2(\sin \rho -\rho \cos \rho)}{\rho}, \qquad
B:=\frac{\rho-\cos \rho \sin \rho}{\sin \rho}.
\end{align*}
 Here one can check that  $A, B> 0$ for $0 < \rho < \pi$ and this will be important later.
We use these and the identities
\begin{align*}
\dot{\rho} = \langle \hat{r}, q \rangle, &\qquad \ddot{\rho} =\frac{1}{\rho}(|q|^2 - \langle \hat{r}, q \rangle^2 )
\end{align*}
 to do the following differentiations and rearrangements:
\begin{align}\label{G-0}
G & = \frac{\rho}{2\sin \rho} A, \\\nonumber
\dot{G} & = \frac{1}{\sin \rho} B \langle \hat{r}, q \rangle ,\\\nonumber
\ddot{G} & = \frac{\rho}{\sin^3 \rho} A \langle \hat{r}, q \rangle^2 +
\frac{1}{\rho \sin \rho}  B (|q|^2- \langle \hat{r}, q \rangle ^2 ),\\\nonumber
\dot{I} & = \frac{2}{\rho}(-\langle \hat{r}, w \rangle \langle q, w\rangle
 + \langle \hat{r}, w\rangle^2 \langle \hat{r}, q\rangle ), \\\nonumber
 \ddot{I} & = \frac{2}{\rho^2}
 (4 \langle \hat r, q\rangle \langle \hat{r}, w\rangle \langle q, w \rangle
 - 4 \langle \hat{r}, w\rangle ^2 \langle \hat{r}, q\rangle^2
 - \langle q, w\rangle^2 + \langle \hat{r}, w \rangle^2 |q|^2).
\end{align}
Key observations here are first,
\begin{align*}
 \ddot G > 0
\end{align*}
and second, the quantities $\dot I $, $\ddot I$ are independent of the normal component $w^\perp$ of $w$ to the plane $\Sigma \subset T_x S^n $ generated by $\hat{r}$ and $q$. Let $w_1 = w- w^\perp$ be the projection of $w$ to $\Sigma$. By separating $I = |w^\perp|^2 + |w_1|^2 - \langle \hat r,  w \rangle^2 $, one sees
\begin{align}\label{H-1}
-\ddot{H} &= \ddot{G} I + 2 \dot{G} \dot{I} + G \ddot{I}\\\nonumber
&= \ddot G |w^\perp|^2 - \ddot H_1 \\\nonumber
& \ge -\ddot H_1
\end{align}
where $H_1$ is the quantity defined by replacing $w$ in \eqref{H-0} with $w_1$, thus independent of $w^\perp$. Notice that the quantity $-\ddot{H_1}$ becomes identical to $-\ddot H$ of $r,q,w_1$ viewed as tangent vectors of the $2$-dimensional round sphere that is the exponential image of $\Sigma$ in the original sphere. This reduces the consideration to two dimensions.

We will need the following key expression obtained by \eqref{G-0} and rearrangement:
\begin{align}\label{H-1-1}
-\ddot{H_1} = \frac{1}{\rho \sin \rho}&  \Big{\{} \big{[}A \frac{\rho^2}{\sin^2 \rho} \langle \hat r, q \rangle^2 \,
+ B (|q|^2-\langle \hat r, q \rangle ^2 )\big{]} (|w_1|^2 - \langle \hat{r}, w_1 \rangle^2 )\\\nonumber
& \  \ +4(B-A)\big{(}\langle \hat r, q \rangle^2 \langle \hat r, w_1 \rangle^2 - \langle \hat r, q \rangle \langle \hat r, w_1 \rangle \langle q , w_1 \rangle \big{)}\\\nonumber
& \ \  + A \big{(} \langle \hat r, w_1 \rangle^2 |q|^2 - \langle q, w_1 \rangle^2\big{)} \Big{\}}.
\end{align}
 Note that
\begin{align}\label{B-A}
B-A= \frac{\rho^2 + \rho \sin \rho \cos \rho - 2 \sin^2 \rho}{\rho \sin \rho} > 0 \hbox{ \ \ for $0 < \rho < \pi$},
\end{align} as can be checked by taking the fourth-order derivative of the numerator.
At this point, Loeper's result in \cite{L2} that $S^n$ is {\bf A3s} can be obtained by
substituting $\langle q, w \rangle = \langle q, w_1 \rangle =0$ into the
expression \eqref{H-1-1} and using the second line of \eqref{H-1}.

\subsection*{Step 2: two dimensional case}
From \eqref{H-1} 
it suffices to show $-\ddot H_1 \ge 0$.
From now on, we assume without loss of generality the dimension is two, and
let $\hat r = (0, 1)$, $q=(\cos \theta , \sin \theta)$, $w_1=(\cos \psi, \sin \psi)$ in $\R^2 \cong T_x S^2$, with $0\le \theta, \psi \le 2 \pi$. Let $T:=\tan \theta$, $S:=\tan \psi$, $ -\infty \le T, S \le + \infty$. One checks from \eqref{H-1-1},
\begin{align}\label{2D H}
-\ddot H_1 = \frac{\cos^2 \theta \cos^2 \psi}{\rho \sin \rho} P,
\end{align}
where
\begin{align*}
P= A \,  S^2 - 2 (2B-A) T \, S + A \frac{\rho^2}{\sin^2 \rho} T^2 + B-A .
\end{align*}
Thus it suffices to show $P > 0$.
$P$ is a convex (since $A > 0$) quadratic polynomial in $S$ with discriminant
\begin{align*}
D:=& 4 (2B-A)^2 T^2 - 4 A (A\frac{\rho^2}{\sin^2 \rho} T^2 +B-A)\\\nonumber
=& 4\Big{\{} \big{(} (2B-A)^2 - A^2 \frac{\rho^2}{\sin^2 \rho} \big{)} T^2 - A(B-A)\Big{\}}.
\end{align*}
We show $D < 0$ (regardless of $T$), which implies $P > 0$. Since $A(B-A) > 0$, $D$ is always negative if
\begin{align*}
0 &\ge (2B-A)^2 - A^2 \frac{\rho^2}{\sin^2 \rho} \\
&= (2B-A + A\frac{\rho}{\sin\rho})(2B-A - A\frac{\rho}{\sin \rho}).
\end{align*}
The first factor is positive, and the second factor is negative since
\begin{align*}
2 B-A - A\frac{\rho}{\sin \rho} =
-\frac{2}{\rho \sin \rho} a(\rho),
\end{align*}
where
\begin{align*}
a(\rho):=\sin^2\rho+ \rho \sin \rho -\rho^2 (1+ \cos \rho)
\end{align*}
which is positive from Lemma~\ref{the function} (since $0 < \rho < \pi$). This establishes
the desired inequality (\ref{almost positive}).

\subsection*{Step 4: equality case}
Let us analyze the cases of equality,  to conclude the almost positive
cross-curvature \eqref{almost positive} of $S^n$.  We only need to show for $0 < \rho < \pi$ that
$-\ddot H =0$ holds if and only if
the three vectors
$q$, $w$, $\hat r$ at $T_x S^n$ are all parallel.
The necessity is easy to verify.
For sufficiency, suppose $-\ddot{H} = 0$. From \eqref{H-1}, $w^\perp =0$. Thus it reduces to two dimensional case as in Step 2. Now, from \eqref{2D H} and $P > 0$, $\cos \theta \cos \psi = 0$. Thus either $q$ or $w$ is parallel to $\hat{r}$. In either case examining with \eqref{H-1-1} shows the other vector is also parallel to $\hat{r}$.
 This establishes almost positivity of the cross-curvature of $S^n$.
\end{proof}



\bibliographystyle{plain}

\end{document}